\documentclass[12pt]{article}

\usepackage[paper=a4paper,left=30mm,right=30mm,top=30mm,bottom=30mm]{geometry}
\usepackage{algorithm}
\usepackage{algpseudocodex}

\usepackage{amsmath}
\usepackage{amssymb} 
\usepackage{amsthm} 
\usepackage{thmtools} 

\usepackage[shortlabels]{enumitem}
    \setlist{nosep}
    \setlist[enumerate,1]{label={(\roman*)}}
\usepackage{todonotes}
\usepackage[font=footnotesize,margin=1cm]{caption}
\usepackage{subcaption}
\usepackage{graphicx}
    \graphicspath{{./code/results/}}

\usepackage{ltablex} 

\ifluatex 
    \usepackage{fontspec}
    \defaultfontfeatures{Ligatures={TeX}}
\else
    \usepackage[utf8]{inputenc}
    \usepackage[T2A,T1]{fontenc}
\fi

\usepackage{mathtools}
\usepackage{microtype}
\usepackage{orcidlink}

\usepackage{pgfplots}
\usepackage{tikz}
    \pgfplotsset{compat=newest}
    \usetikzlibrary{plotmarks}
    \usetikzlibrary{calc}
    \usetikzlibrary{positioning}
    \usetikzlibrary{arrows}
    \usetikzlibrary{arrows.meta}
    \usetikzlibrary{graphs}
    \usetikzlibrary{matrix}
    \usepgfplotslibrary{patchplots}
     
\usepackage{xspace}
\usepackage{xcolor} 
\usepackage{cite}
\usepackage{hyperref}
\usepackage[nameinlink]{cleveref}

\theoremstyle{plain}
\newtheorem{theorem}{Theorem}

\newtheorem{definition}[theorem]{Definition}
\newtheorem{example}[theorem]{Example}
\newtheorem{lemma}[theorem]{Lemma}
\newtheorem{proposition}[theorem]{Proposition}
\newtheorem{remark}[theorem]{Remark}
\newtheorem{scheme}[theorem]{Scheme}
\newtheorem{assumption}[theorem]{Assumption}

\definecolor{darkgreen}{RGB}{15, 135, 0}

\newcommand{\norm}[1]{\Vert #1 \Vert}

\newcommand{\RR}{\mathbb{R}}
\newcommand{\NN}{\mathbb{N}}
\newcommand{\RRnonneg}{[0,\infty)}

\DeclareMathOperator{\Span}{span}

\newcommand{\Uad}{\mathcal{U}_{\mathrm{ad}}}

\newcommand{\tp}{^{\hspace{-1pt}\mathsf{T}}}

\newcommand{\mtp}{^{-\mathsf{T}}}

\AtBeginDocument{\renewcommand{\tilde}{\widetilde}}
\AtBeginDocument{\renewcommand{\hat}{\widehat}}

\newcommand{\dt}{\,\mathrm{d}t}

\newcommand{\ds}{\,\mathrm{d}s}

\newcommand{\ddtheta}{\frac{\mathrm{d}}{\mathrm{d}\theta}}

\newcommand{\portHamiltonian}{port-Ha\-mil\-to\-ni\-an\xspace}

\renewcommand{\vec}[1]{\begin{bmatrix} #1 \end{bmatrix}}
\newcommand{\textvec}[1]{[\begin{smallmatrix} #1 \end{smallmatrix}]}

\newcommand{\domain}{\mathcal{D}}
\newcommand{\B}{g} 
\newcommand{\D}{k} 
\newcommand{\Ucal}{\mathcal{U}}
\newcommand{\ham}{\mathcal{H}}
\renewcommand{\eta}{\nabla\ham}

\newcommand{\Qbf}{\mathbf{Q}} 
\newcommand{\Sbf}{\mathbf{S}} 
\newcommand{\Rbf}{\mathbf{R}} 

\newcommand{\fbar}{\overline{f}}
\newcommand{\etabar}{\overline{\nabla}\ham}
\newcommand{\Bbar}{\overline{\B}}
\newcommand{\Dbar}{\overline{\D}}

\newcommand{\ubar}{\overline{u}}
\newcommand{\ybar}{\overline{y}}
\newcommand{\ellbar}{\overline{\ell}}
\newcommand{\ntimesteps}{q}
\newcommand{\gammabar}{\overline{\gamma}}
\newcommand{\Wbar}{\overline{W}}
\newcommand{\matrixproj}{\mathcal{P}}
\newcommand{\dgfixpoint}{\mathsf{F}}
\newcommand{\dgnonzeroset}{\mathsf{M}}
\newcommand{\eye}{I}

\newcommand{\lebesgue}{L}

\newcommand{\hamc}{\hat{\mathcal{H}}}
\newcommand{\etac}{\nabla\hamc}
\newcommand{\yc}{\hat{y}}

\newcommand{\error}{\mathcal{E}}
\newcommand{\errorpower}{\error_{\mathrm{power}}}

\newcommand{\errorstatenodal}{\error_{\mathrm{state,nodal}}}

\pgfplotsset{
  custom y log style/.style={
      yticklabel={
        \pgfkeys{/pgf/fpu=true}
        \pgfmathparse{exp(\tick)}
        \pgfmathprintnumber[fixed,fixed zerofill, precision=2]{\pgfmathresult}
        \pgfkeys{/pgf/fpu=false}
      }
  }
}

\allowdisplaybreaks

\hypersetup{
        colorlinks=false,
        pdfborder={0,0,0},
        pdfborderstyle={/S/U/W 0},
        }

\title{A discrete gradient scheme for preserving QSR-dissipativity}

\author{
    Attila Karsai\,\orcidlink{0000-0001-9905-433X}\thanks{Institute of Mathematics, EPFL, Switzerland},
    Philipp Schulze\,\orcidlink{0000-0002-7299-4628}\thanks{Institute of Mathematics, Technische Universität Berlin, Germany}\,\,\thanks{Institute of Mathematics, University of Potsdam, Germany}
}

\providecommand{\keywords}[1]
{
  \small	
  \textbf{Keywords:} #1\\
}

\providecommand{\ams}[1]
{
  \small	
  \textbf{AMS subject classifications:} #1
}

\date{\today}

\colorlet{g}{black!40!green}
\colorlet{r}{black!30!red}
\colorlet{b}{black!20!blue}
\definecolor{y}{RGB}{253,229,65}
\colorlet{o}{black!10!orange}
\colorlet{p}{black!30!purple}

\allowdisplaybreaks 

\begin{document}

\renewcommand{\hbar}{\overline{h}}

\maketitle

\begin{abstract}
	The notion of dissipative dynamical systems provides a formal description of processes that cannot generate energy internally.
    For these systems, changes in energy can only occur due to an external energy supply or dissipation effects.
    Unfortunately, dissipative properties tend to deteriorate in numerical computations, especially in nonlinear systems.
    Discrete gradient methods can help mitigate this problem.
    In this paper, we present a class of structure-preserving time discretization schemes based on discrete gradients for a special class of systems that are dissipative with respect to a quadratic supply rate.
\end{abstract}
\keywords{nonlinear systems, dissipative systems, time discretization, structure-preservation, quadratic supply rate}
\\\noindent\ams{37M15, 65P10, 93-10, 93C10}

\section{Introduction}

Dynamical systems arising from physical or other applications typically feature some relevant qualitative properties which are worth to be preserved when discretizing the system in time.
Examples for such properties are stability, passivity, or conserved quantities.
In general, time discretization does not necessarily preserve these properties, which may lead to unphysical results, such as planets leaving their orbit, cf.~\cite[Sec.~I.2]{hairer06-geometric}.
Often the qualitative properties are related to an algebraic or geometric structure of the dynamical system and, therefore, structure-preserving time discretization schemes may be used for the preservation of such properties.
Examples for relevant structures are the Hamiltonian \cite{arnold89-mathematical}, gradient \cite{hirsch74-differential}, GENERIC \cite{grmela97-dynamics,oettinger97-dynamics}, or port-Hamiltonian \cite{vanderschaft14-port} ones.
Among these, port-Hamiltonian systems are especially relevant for this paper since this structure ensures passivity, which is a special case of QSR-dissipativity, i.e., dissipativity w.r.t.~a quadratic supply rate.

In this paper we propose a time discretization scheme for a certain class of nonlinear QSR-dissipative systems, which ensures that a time-discrete analogue of the time-continuous power balance is satisfied.
To this end, we employ discrete gradients, which are also used, e.g., for the structure-preserving discretization of Hamiltonian or port-Hamiltonian systems, cf.~\cite{gonzalez96-time,mclachlan99-geometric,celledoni17-energy,kinon26-discrete}.
Other structure-preserving time discretization schemes for port-Hamiltonian systems are based on Petrov--Galerkin projections \cite{morandin24-modeling,giesselmann25-energy-consistent} or Gauss collocation methods \cite{kotyczka19-discrete-time,mehrmann19-structure-preserving}.
For further structure-preserving schemes for Hamiltonian, gradient, and GENERIC systems, we refer to \cite{hairer06-geometric,juengel19-structure-preserving,egger21-energy,oettinger18-generic,juengel22-minimizing-movements} and the references therein.
While structure-preserving schemes for port-Hamiltonian systems allow obtaining time-discrete power balances for passive systems, there are currently no such time discretization schemes available for the more general class of QSR-dissipative systems.

The main contributions of this paper are listed in the following.
\begin{itemize}
	\item We propose a class of time discretization schemes for special QSR-dissipative systems with equal input and output dimensions, cf.~\Cref{scheme:dg-qsr}.
		To this end, we first demonstrate that the considered class of QSR-dissipative systems admits a special structure on the time-continuous level involving the gradient of the storage function, cf.~\cite{ortega02-interconnection,wang03-generalized} for similar ideas in the context of port-Hamiltonian systems.
		The proposed time discretization schemes enforce a corresponding structure on the time-discrete level via discrete gradients and suitable projections.
        We show that solutions of the time-discrete system satisfy a time-discrete power balance, cf.~\Cref{prop:discr:dg-qsr-powerbalance}.
	\item We provide sufficient conditions for the time discretization scheme to be locally well-posed and second-order consistent, cf.~\Cref{thm:discr:qsr-dg-convergence}.
		By a standard stability argument for one-step schemes, second-order consistency implies second-order accuracy of the numerical solution, cf.~\Cref{rem:global-error}.
	\item We illustrate the approach by means of several numerical examples of QSR-dissipative systems and observe second order accuracy as well as the satisfaction of the time-discrete power balance up to the precision of the nonlinear system solver applied in every time step, cf.~\Cref{sec:numerical_experiments}.
\end{itemize}

The rest of this paper is structured as follows.
In \Cref{sec:problem_setting} we present the problem setting and some relevant definitions and preliminary results on QSR-dissipative systems.
\Cref{sec:time_discretization} is dedicated to the proposed time discretization scheme and its properties, and in \Cref{sec:numerical_experiments} we illustrate the approach by means of multiple numerical experiments.
In Appendix~\ref{app:one-step}, we collect some general results on local well-posedness, consistency, and accuracy of one-step schemes, which are used in \Cref{sec:time_discretization}.

\paragraph{Notation}
Throughout this manuscript, the following notation is used.
We denote the identity matrix of size $n\times n$ by $\eye_n$, and omit the subscript~$n$ if the size is clear from the context.
Moreover, $\norm{\cdot}_{\RR^n}$ denotes the Euclidean norm, and similarly the subscript $\RR^n$ is omitted if the dimension of the space is clear from the context.
The set of scalar multiples of the vector $v\in\RR^n$ is written as $\Span \lbrace v\rbrace$, and $A\tp$ refers to the transpose of $A \in \RR^{n,m}$.
The sets of continuous and $k$-times continuously differentiable functions from~$[0,T]$, $T>0$ to~$V\subset\RR^n$ are denoted by $C([0,T],V)$ and $C^k([0,T],V)$, respectively.
Furthermore, $\partial_i$ denotes the partial derivative w.r.t.~the $i$-th variable, and we write $\dot z$ for the time derivative of $z \colon [0,T] \to \RR^n$.

\section{Problem setting}\label{sec:problem_setting}

\subsection{Dynamical systems}

In this paper, we consider nonlinear finite-dimensional dynamical system of the form
\begin{equation}\label{eq:moylan-system}
    \begin{aligned}
        \dot{z} & = f(z) + \B(z) u, ~~ z(0) = z_0 \\
        y & = h(z) + \D(z) u.
    \end{aligned}
\end{equation}
Here,~$f \colon \RR^n \to \RR^n$,~$\B \colon \RR^n \to \RR^{n,m}$,~$h \colon \RR^n \to \RR^m$,~$\D \colon \RR^n \to \RR^{m,m}$ model the dynamics and are assumed continuous, and~$z_0 \in \RR^n$ is the initial condition.
Moreover, the function $u \in \Ucal = C(\RR, \RR^m)$ models external control inputs.%
\footnote{Defining controls on the whole real line is a purely technical assumption, which ensures that the domains of well-posedness of the time-discrete scheme in \Cref{sec:time_discretization} are open sets in $\RR^n \times \RR \times \RR$, cf.~Appendix~\ref{app:one-step}.}
We make the following assumption on well-posedness of~\eqref{eq:moylan-system} throughout the manuscript.
Here, a classical solution is a continuously differentiable function $z$ which satisfies the differential equation in~\eqref{eq:moylan-system} for all~$t\in[0,T]$, $T>0$.

\begin{assumption}\label{as:well-posedness}
	For any $u\in\Ucal$ and $z_0\in\RR^n$, the system \eqref{eq:moylan-system} has a unique classical solution on the time horizon~$[0,T]$,~$T>0$.
\end{assumption}

\subsection{Dissipativity}
Energy properties of the system~\eqref{eq:moylan-system} can be modeled as introduced in the seminal work~\cite{willems72-dissipative1}.
The precise definition, adapted to our setting, reads as follows.

\begin{definition}[dissipative systems~\cite{willems72-dissipative1}]\label{def:willems-dissipative}
    Consider the system~\eqref{eq:moylan-system}.
    In addition to \Cref{as:well-posedness}, let $s \colon \RR^m \times \RR^m \to \RR$ be such that the integral~$\int_{t_0}^{t_1} s(u(t), y(t)) \dt$ is well-defined for all timepoints~$0 \leq t_0 \leq t_1 \leq T$, initial values~$z(t_0) \in \RR^n$ and controls~$u \in \Ucal$.
    Then, the system is called \emph{dissipative with respect to the supply rate~$s$} if there exists $\ham \colon \RR^n \to \RRnonneg$ satisfying
    \begin{equation}\label{eq:model:willems-dissipation-inequality}
        \ham(z(t_1))
        \leq
        \ham(z(t_0))
        +
        \int_{t_0}^{t_1}
            s(u(t), y(t))
        \dt
    \end{equation}
    for all~$0 \leq t_0 \leq t_1$,~$z(t_0) \in \RR^n$ and~$u \in \Ucal$.
    In that case, we call $\ham$ a \emph{storage function} and \eqref{eq:model:willems-dissipation-inequality} the corresponding \emph{dissipation inequality}.
\end{definition}

Loosely speaking, systems dissipative in the sense of \Cref{def:willems-dissipative} are associated with an energy functional that can only increase due to interaction with the environment via inputs and outputs.
Let us shortly comment on some important special cases of \Cref{def:willems-dissipative}.
First, if~$\theta\mapsto\ham(z(\theta))$ is assumed to be differentiable, then the dissipation inequality~\eqref{eq:model:willems-dissipation-inequality} is equivalent to
\begin{equation}\label{eq:willems-differential-dissipation-inequality}
    \ddtheta \ham(z(\theta)) \Big|_{\theta = t} \leq s(u(t), y(t))
\end{equation}
for all~$t\in[0,T]$,~$z_0 \in \RR^n$ and~$u \in \Ucal$.
Second, if the amount of dissipated energy is known, then the inequalities~\eqref{eq:model:willems-dissipation-inequality} and~\eqref{eq:willems-differential-dissipation-inequality} can be modified to hold with equality, leading to the notion of \emph{dissipation rates}.

\begin{definition}[dissipation rate~\cite{willems72-dissipative1}]\label{def:willems-dissipation-rate}
    In addition to the assumptions of \Cref{def:willems-dissipative}, let~$d \colon \RR^n \times \RR^m \to \RRnonneg$ be such that~$\int_{t_0}^{t_1} d(z(t), u(t)) \dt$ is well-defined for all~$0\leq t_0 \leq t_1$,~$z(t_0) \in \RR^n$ and~$u \in \Ucal$.
    If for all~$0\leq t_0 \leq t_1$,~$z(t_0) \in \RR^n$ and~$u \in \Ucal$ the \emph{energy balance}
    \begin{equation}\label{eq:willems-energy-balance}
        \ham(z(t_1))
        =
        \ham(z(t_0))
        +
        \int_{t_0}^{t_1}
            s(u(t), y(t))
            -
            d(z(t), u(t))
        \dt
    \end{equation}
    is satisfied, then~$d$ is called the \emph{dissipation rate} of the dynamical system with supply rate~$s$ and storage function~$\ham$.
\end{definition}

Analogous to~\eqref{eq:willems-differential-dissipation-inequality}, if~$\theta\mapsto\ham(z(\theta))$ is differentiable, then~\eqref{eq:willems-energy-balance} is equivalent to the \emph{power balance}
\begin{equation}\label{eq:willems-power-balance}
    \ddtheta \ham(z(\theta)) \Big|_{\theta = t} = s(u(t), y(t)) - d(z(t), u(t))
\end{equation}
for all~$t\in[0,T]$, $z_0 \in \RR^n$ and~$u \in \Ucal$.

We are particularly interested in quadratic supply rates.
These take the form
\begin{equation}\label{eq:hill-moylan-qsr-supply}
    s(u,y) = y\tp \Qbf y + 2 y\tp \Sbf u + u\tp \Rbf u,
\end{equation}
where the matrices~$\Qbf, \Sbf, \Rbf \in \RR^{m,m}$ satisfy~$\Qbf=\Qbf\tp$ and~$\Rbf=\Rbf\tp$.
If the system~\eqref{eq:moylan-system} is dissipative with respect to the supply rate~\eqref{eq:hill-moylan-qsr-supply}, we call it \emph{QSR-dissipative}.
Examples of supply rates modeled by~\eqref{eq:hill-moylan-qsr-supply} are the \emph{impedance supply} $s(u,y) = y\tp u$, which leads to the notion of \emph{passive} systems, and the \emph{scattering supply} $\norm{u}^2 - \norm{y}^2$, which is related to the notion of \emph{finite $L^2$-gain}, see, e.g.,~\cite{vanderschaft17-l2gain}.%

\subsection{Characterizations of dissipative systems}

For linear time-invariant systems, popular characterizations of dissipative dynamical system include the Lur'e equations, see e.g.~\cite[Sec.~3--4]{brogliato20-dissipative}.
The following result is due to~\cite{hill75-cyclo} and a generalization of the Lur'e equations to nonlinear systems.

\begin{theorem}[\!{\cite[Theorem 8]{hill75-cyclo}}]\label{thm:hill-moylan-result}
    Let \Cref{as:well-posedness} hold with $T = \infty$, and let~$\ham \colon \RR^n \to \RRnonneg$ be continuously differentiable.
    Then~$\ham$ is a storage function for~\eqref{eq:moylan-system} with respect to the supply rate~\eqref{eq:hill-moylan-qsr-supply} if and only if there exists~$p \in \NN$ and functions~$\ell \colon \RR^n \to \RR^p$,~$W \colon \RR^n \to \RR^{p,m}$ such that
    \begin{equation}\label{eq:moylan-qsr-conditions}
        \begin{aligned}
        \eta(z)\tp f(z) & = h(z)\tp \Qbf h(z) - \ell(z)\tp \ell(z), \\
        \tfrac12 \eta(z)\tp \B(z) & = h(z)\tp (\Qbf \D(z) + \Sbf) -  \ell(z)\tp W(z), \\
        W(z)\tp W(z) & = \Rbf + \D(z)\tp \Sbf + \Sbf\tp \D(z) + \D(z)\tp \Qbf \D(z)
        \end{aligned}
    \end{equation}
    for all~$z \in \RR^n$.
\end{theorem}

\section{Time discretization}\label{sec:time_discretization}

Our main goal is to discretize~\eqref{eq:moylan-system} in time.
More precisely, for given time points $t_i \in [0,T]$, $i=0,\dots,\ntimesteps$ with
\begin{equation*}
    0 = t_0 < \dots < t_{\ntimesteps} = T,\quad \tau_i\vcentcolon= t_{i+1}-t_i\text{ for }i=0,\ldots,\ntimesteps-1,
\end{equation*}
we aim to find approximations $z_i \approx z(t_i)$, where~$z$ denotes the true trajectory of~\eqref{eq:moylan-system}.
While many different methods for this task exist, see, e.g.,~\cite{hairer93-solving} for an overview, dissipativity properties are not preserved under discretization in general.
Our goal is to define a discretization scheme that guarantees a time-discrete analogue of the power balance~\eqref{eq:willems-power-balance}, see the upcoming \Cref{prop:discr:dg-qsr-powerbalance} for the precise notion.
For an overview on structure-preserving discretization methods, we refer to~\cite{hairer06-geometric} and~\cite{kotyczka19-numerical}.
In the following, we focus on \emph{discrete gradient methods}.

\subsection{Discrete gradients}

Discrete gradients can be understood as a generalization of difference quotients to higher dimensions.
The precise definition is as follows.

\begin{definition}\label{def:discrete-gradient}
    Let~$\ham \colon \RR^n \to \RR$ be continuously differentiable.
    We call a continuous function~$\etabar \colon \RR^n \times \RR^n \to \RR^n$ a \emph{discrete gradient of~$\ham$} if it satisfies the \emph{mean value property}
    \begin{equation}\label{eq:discr:dg-mean-value-property}
        \ham(w) - \ham(z) = \etabar(z, w)\tp(w - z)
    \end{equation}
    and the \emph{consistency property}
    \begin{equation}\label{eq:discr:dg-consistency-property}
        \etabar(z,z) = \eta(z)
    \end{equation}
    for all~$z, w \in \RR^n$.
\end{definition}

If~$n=1$, then for~$z \neq w$ the unique discrete gradient is the difference quotient
\begin{equation*}
    \etabar(z, w) = \frac{\ham(w) - \ham(z)}{w - z}.
\end{equation*}
In higher dimensions, the discrete gradient is not unique, since only the component along~$w - z$ is restricted.
As noted by~\cite{mclachlan99-geometric}, this restriction results in the fact that the approximation of point values of $\nabla\ham$ by means of discrete gradients can be at most of second order.
Nevertheless, the overall discretization scheme can have higher order, see e.g.~\cite{norton14-discrete,celledoni17-energy,eidnes22-order,kemmochi23-higher} for corresponding approaches.
A characterization of discrete gradients is given in~\cite[Proposition 3.2]{mclachlan99-geometric}.

In the literature, different examples of discrete gradients~$\etabar$ can be found, which we briefly summarize in the following.
The \emph{mean value discrete gradient}~\cite{harten83-upstream} is used in average vector field methods, see, e.g.,~\cite{celledoni12-preserving}, and reads
\begin{equation*}
    \etabar(z, w)
    =
    \int_{0}^{1} \eta\big( (1-s) z + s w \big) \ds.
\end{equation*}
The \emph{Itoh--Abe discrete gradient}~\cite{itoh88-hamiltonian} is given by
\begin{equation*}
    \etabar(z, w)
    =
    \vec{
        \frac{
            \ham(w^{(1)}, z^{(2)}, \cdots, z^{(n)})
            -
            \ham(z)
            }{
                w^{(1)} - z^{(1)}
            }
        \\
        \frac{
            \ham(w^{(1)}, w^{(2)}, z^{(3)}, \cdots, z^{(n)})
            -
            \ham(w^{(1)}, z^{(2)}, \cdots, z^{(n)})
            }{
                w^{(2)} - z^{(2)}
            }
        \\
        \vdots
        \\
        \frac{
            \ham(w)
            -
            \ham(w^{(1)}, \cdots, w^{(n-1)}, z^{(n)})
            }{
                w^{(n)} - z^{(n)}
            }
    },
\end{equation*}
where~$z^{(k)}$ denotes the~$k$-th entry of~$z \in \RR^n$ and~$0/0$ is interpreted as~$\partial_i \ham(z)$.
The \emph{Gonzalez discrete gradient}~\cite{gonzalez96-time} is given by
\begin{equation}\label{eq:discr:gonzalez-dg}
    \etabar(z, w)
    =
    \eta \left( \frac{z + w}2 \right)
    +
    \frac{
        \ham(w) - \ham(z)
        -
        \eta \left( \frac{z + w}2 \right)\tp (w - z)
    }
    {
        \norm{w - z}^2
    }
    (w - z)
\end{equation}
for~$z \neq w$ and~$\etabar(z, z) = \eta(z)$.

Apart from applications in structure-preserving discretization~\cite{harten83-upstream,itoh88-hamiltonian,gonzalez96-time,mclachlan99-geometric}, discrete gradient methods are applied to optimization methods~\cite{ehrhardt25-geometric}.
They can also be extended to systems with algebraic equations~\cite{schulze23-structure2,kinon26-discrete}.
In what follows, we discuss their application to QSR-dissipative systems.

\subsection{Proposed discretization scheme}

In this section, we show how discrete gradient methods can be used to obtain a time-discrete analogue of the power balance~\eqref{eq:willems-power-balance} for QSR-dissipative systems.
We make the following assumption.

\begin{assumption}\label{as:qsr-scheme}
    \Cref{as:well-posedness} is satisfied.
    Moreover, the system~\eqref{eq:moylan-system} is dissipative with respect to the supply rate~\eqref{eq:hill-moylan-qsr-supply} and the continuously differentiable storage function~$\ham \colon \RR^n \to \RR$ satisfies the conditions~\eqref{eq:moylan-qsr-conditions}.
    In addition, the matrix $\Qbf \D(z) + \Sbf$ is invertible for all~$z\in\RR^n$.
\end{assumption}

We emphasize that the invertibility assumption on $\Qbf \D(z) + \Sbf$ restricts the class of QSR-dissipative systems.
For instance, systems where the numbers of inputs and outputs do not coincide are excluded by the assumption.
Among others, the assumption is satisfied for passive systems, i.e., for the supply rate ~$s(u,y) = y\tp u$ with~$\Qbf = 0$,~$\Sbf=\tfrac12\eye_m$ and~$\Rbf = 0$.
There are also many other relevant examples where the assumption is satisfied, e.g., the ones considered in \Cref{sec:numerical_experiments}.

If \Cref{as:qsr-scheme} holds, then the system~\eqref{eq:moylan-system} is QSR-dissipative with storage function~$\ham$ with the dissipation rate~$d(z, u) = \norm{\ell(z) + W(z)u}^2$ in~\eqref{eq:willems-power-balance}, see the proof of \Cref{thm:hill-moylan-result} in~\cite{hill75-cyclo}.
Hence, the power balance~\eqref{eq:willems-power-balance} reads
\begin{equation}\label{eq:discr:qsr-powerbalance}
    \ddtheta \ham(z(\theta)) \Big|_{\theta = t}
    =
    \eta(z)\tp (f(z) + \B(z) u)
    =
    s(u,y) - \norm{\ell(z) + W(z)u}^2,
\end{equation}
where we abbreviate~$z = z(t)$,~$u = u(t)$ and~$y = y(t)$.

As mentioned before, our goal is to state a scheme that produces iterates which satisfy a discrete power balance resembling~\eqref{eq:discr:qsr-powerbalance}.
For this, we note that \Cref{as:qsr-scheme} implies that~$h(z)$ can be recovered from the second equation in~\eqref{eq:moylan-qsr-conditions} as
\begin{equation}\label{eq:discr:qsr-h-def}
    h(z) = (\Qbf \D(z) + \Sbf)\mtp(\tfrac12 \B(z)\tp \eta(z) + W(z)\tp \ell(z)).
\end{equation}
To discretize the system~\eqref{eq:moylan-system}, we split the dynamics~$f(z)$ into two components: one component in the direction of~$\eta(z)$, and another component orthogonal to~$\eta(z)$.
If~$\eta(z)\neq0$, then this can be achieved using the orthogonal projections~$\matrixproj_{\eta(z)}$ onto~$\Span\{\eta(z)\}$ and~$\matrixproj_{\eta(z)^\perp}$ onto~$\Span\{\eta(z)\}^\perp$ via
\begin{equation}\label{eq:discr:f-qsr-decomposition}
    f(z)
    =
    \matrixproj_{\eta(z)} f(z) + \matrixproj_{\eta(z)^\perp} f(z).
\end{equation}
Here, for arbitrary~$v\neq 0$ we set~$\matrixproj_v \coloneq \frac{v v\tp}{\norm{v}^2}$ and~$\matrixproj_{v^\perp} \coloneq \eye - \matrixproj_v$.
This decomposition also appears in~\cite{wang03-generalized}, where the authors call it the \emph{orthogonal decomposition Hamiltonian realization} of~$f$.
Using~\eqref{eq:moylan-qsr-conditions}, for~$\eta(z)\neq 0$ we obtain
\begin{equation}\label{eq:discr:f-qsr-first-term}
    \matrixproj_{\eta(z)} f(z)
    =
    \frac{ \eta(z)\tp f(z)}{\norm{\eta(z)}^2} \eta(z)
    =
    \frac{h(z)\tp \Qbf h(z) - \norm{\ell(z)}^2}{\norm{\eta(z)}^2} \eta(z)
    =
    \gamma(z) \eta(z)
\end{equation}
where
\begin{equation*}
    \gamma(z) \coloneq \frac{h(z)\tp \Qbf h(z) - \norm{\ell(z)}^2}{\norm{\eta(z)}^2}
\end{equation*}
with~$h(z)$ being expressed as in~\eqref{eq:discr:qsr-h-def}.
Replacing~$f$,~$g$,~$u$,~$k$,~$\ell$,~$W$ by suitable approximations, adapting~$h$ and~$\gamma$ according to~\eqref{eq:discr:qsr-h-def}, and replacing~$\eta$ by a discrete gradient~$\etabar$, we arrive at the following class of discretization schemes.
A specific discretization scheme can be obtained, e.g., by using midpoint approximations%
\footnote{Here, ``midpoint approximation'' refers to $\fbar(z,w) = f(\tfrac{z+w}{2})$ and $\bar{u}(t;\tau)=u(t+\tfrac\tau2)$.}
for $f$,~$g$,~$u$,~$k$,~$\ell$,~$W$ and the Gonzalez discrete gradient for $\etabar$.

\begin{scheme}[discrete gradient method for QSR-dissipative systems]\label{scheme:dg-qsr}
    Let \Cref{as:qsr-scheme} hold and let~$\etabar$ be a discrete gradient of~$\ham$.
    Moreover, let $\fbar\colon \RR^n \times \RR^n \to \RR^n$,~$\Bbar \colon \RR^n \times \RR^n \to \RR^{n,m}$,~$\Dbar\colon \RR^n \times \RR^n \to \RR^{m,m}$,~$\ellbar\colon \RR^n \times \RR^n \to \RR^p$,~$\Wbar\colon \RR^n \times \RR^n \to \RR^{p,m}$ be continuous and consistent in the sense $\fbar(z,z)=f(z)$, $\Bbar(z,z)=\B(z)$, $\ldots$, $\Wbar(z,z)=W(z)$ for all $z\in\RR^n$, and satisfy
    \begin{align*}
    	\det(\Qbf \Dbar(z, w) + \Sbf) &\ne 0,\\
    	\Wbar(z,w)\tp \Wbar(z,w) &= \Rbf + \Dbar(z,w)\tp \Sbf + \Sbf\tp \Dbar(z,w) + \Dbar(z,w)\tp \Qbf \Dbar(z,w)
    \end{align*}
    for all $(z,w)\in\RR^n \times \RR^n$.
    Furthermore, according to~\eqref{eq:discr:qsr-h-def}, define~$\hbar \colon \RR^n \times \RR^n \to \RR^m$ via
    \begin{equation}\label{eq:discr:qsr-dg-hbar-definition}
        \hbar(z, w)
        \coloneq
        (\Qbf \Dbar(z, w) + \Sbf)\mtp
        \left(
            \tfrac12 \Bbar(z, w)\tp \etabar(z, w)
            +
            \Wbar(z, w)\tp \ellbar(z,w)
        \right)
    \end{equation}
    and let $\ubar\colon \RR\times\RR\to\RR^m$ be continuous and satisfy the consistency condition $\ubar(t;0)=u(t)$ for all $t\in \RR$.
    Finally, for~$z, w \in \RR^n$ with~$\etabar(z, w) \neq 0$, define
    \begin{equation*}
        \gammabar(z, w) \coloneq
        \frac
        {
            \hbar(z, w)\tp \Qbf \hbar(z, w)
            -
            \norm{\ellbar(z, w)}^2
        }
        {
            \norm{\etabar(z, w)}^2
        }.
    \end{equation*}
    Then, we consider the following time-discrete problem: For given $z_i\in\RR^n$, find~$z_{i+1} \in \RR^n$ such that~$\etabar(z_i, z_{i+1}) \neq 0$ and
    \begin{equation}\label{eq:discr:dg-qsr-discrete-system}
    \begin{aligned}
        \frac{ z_{i+1}  - z_{i} }
        {\tau_i}
        &
        =
        \gammabar(z_i,z_{i+1}) \etabar(z_i,z_{i+1})
        +
        \matrixproj_{\etabar(z_i,z_{i+1})^\perp}
        \fbar(z_i,z_{i+1})
        +
        \Bbar(z_i,z_{i+1}) \ubar(t_i;\tau_i)
    \end{aligned}
    \end{equation}
    for~$i=0,\dots,\ntimesteps-1$.
\end{scheme}

Next, we provide sufficient conditions for~\eqref{eq:discr:dg-qsr-discrete-system} being locally well-posed and for the resulting one-step scheme being second-order consistent.
To this end, we view \Cref{scheme:dg-qsr} as a one-step scheme in the sense of \Cref{def:locally-well-posed}: the one-step map~$\Psi$ is defined implicitly by the discrete system~\eqref{eq:discr:dg-qsr-discrete-system} on an open domain of well-posedness, where the condition $\etabar(z, w) \neq 0$ restricts the domain to an open set.
In Appendix~\ref{app:one-step}, we collect general results on local well-posedness, consistency, and accuracy of one-step schemes; these results are standard in the literature, see, e.g.,~\cite[Theorem~3.2 in Section~II.3]{hairer06-geometric} and~\cite[Chapter~4]{deuflhard02-scientific}, and there we present self-contained, problem-adapted versions.
The general strategy is as follows.
Local well-posedness is established by means of the implicit function theorem, which yields a twice continuously differentiable one-step map.
The scheme is then shown to be symmetric and consistent, i.e., its increment function reproduces the vector field $f(z) + \B(z) u(t)$.
Finally, the symmetry of the scheme is exploited: a consistent and symmetric one-step method is automatically second-order consistent, cf.~\Cref{lem:order-symmetric-methods}.
Together with the classical stability result of \Cref{rem:global-error}, this implies second-order accuracy of the scheme.

Under suitable smoothness assumptions, these conditions are satisfied, e.g., when choosing midpoint approximations for $f$,~$\B$,~$\D$,~$\ell$,~$W$,~$u$ and the mean value or Gonzalez discrete gradient for $\etabar$.
Here and in the following, we separate the stepsize $\tau$ from the state and time arguments by a semicolon, i.e., we write $\Psi(z, t; \tau)$ for the value of the one-step map at state $z$, time $t$, and stepsize $\tau$

\begin{theorem}[well-posedness and order]\label{thm:discr:qsr-dg-convergence}
    Let \Cref{as:qsr-scheme} hold.
    Assume that~$f,\B,\D,u$ in~\eqref{eq:moylan-system} and~$\ell,W$ in~\eqref{eq:moylan-qsr-conditions} are twice continuously differentiable, that the approximations~$\fbar$,~$\Bbar$,~$\Dbar$,~$\ellbar$,~$\Wbar$,~$\etabar$ are symmetric and twice continuously differentiable, and that $\ubar$ is twice continuously differentiable and time-symmetric, i.e., $\ubar(t;\tau) = \ubar(t+\tau;-\tau)$ for all $t,\tau\in\mathbb{R}$.
    Then \Cref{scheme:dg-qsr} is locally well-posed in the sense of \Cref{def:locally-well-posed}.

    Moreover, if the exact solution~$z \in C^3([0,T],\RR^n)$ of~\eqref{eq:moylan-system} with initial value~$z_0 \in \RR^n$ and control~$u$ satisfies~$\eta(z(t)) \neq 0$ for all~$t \in [0,T]$, then \Cref{scheme:dg-qsr} is second-order accurate in the sense of \Cref{def:accuracy}.
\end{theorem}

\begin{proof}
    We split the proof into three parts: we first establish local well-posedness in the sense of \Cref{def:locally-well-posed} by means of the implicit function theorem, we then verify the symmetry condition, and finally we apply the general results of Appendix~\ref{app:one-step}, which yield that the scheme is second-order consistent; consistency itself follows directly from the construction of the scheme.

    \paragraph{Local well-posedness}
    Since matrix inversion is a smooth operation on the set of invertible matrices and since $\Dbar$,~$\Bbar$,~$\etabar$,~$\Wbar$,~$\ellbar$ are twice continuously differentiable, the function~$\hbar$ defined by~\eqref{eq:discr:qsr-dg-hbar-definition} is twice continuously differentiable.
    Consequently, also $\gammabar$ and $\matrixproj_{\etabar}$ are twice continuously differentiable on the set
    \begin{equation*}
        \dgnonzeroset
        \coloneq
        \big\{
            (z,w) \in \RR^n \times \RR^n
        ~\big|~
            \etabar(z,w) \neq 0
        \big\}.
    \end{equation*}
    Define
    \begin{equation*}
        \begin{aligned}
            \dgfixpoint \colon \domain(\dgfixpoint) &\to \RR^n, \\
            \big( (z, t; \tau), w \big) &\mapsto
            w - z - \tau
            \Big(
                \gammabar(z, w) \etabar(z,w)
                +
                \matrixproj_{\etabar(z,w)^\perp} \fbar(z, w)
                +
                \Bbar(z,w) \ubar(t;\tau)
            \Big),
        \end{aligned}
    \end{equation*}
    with domain~$\domain(\dgfixpoint) \coloneq \{ ((z, t; \tau), w) \in \RR^{n+2}\times\RR^n ~|~ (z,w) \in \dgnonzeroset \}$.
    Then the discrete system~\eqref{eq:discr:dg-qsr-discrete-system} is equivalent to $\dgfixpoint\big((z,t;\tau), w\big) = 0$ with $w = z_{i+1}$, $z = z_i$, $\tau = \tau_i$ and $t = t_i$.
    By our assumptions, the map~$\dgfixpoint$ is twice continuously differentiable.
    Furthermore, set
    \begin{equation*}
        \tilde\dgnonzeroset
        \coloneq
        \big\{
            z \in \RR^n
        ~\big|~
            \eta(z)\ne 0
        \big\}.
    \end{equation*}
    Since $\etabar(z,z) = \eta(z)$, we have~$\dgfixpoint\big((z_0, t_0; 0), z_0\big) = 0$ and~$D_2 \dgfixpoint\big((z_0, t_0; 0), z_0\big) = \eye_n$ for all~$z_0\in\tilde\dgnonzeroset$ and~$t_0 \in \RR$, where~$D_2 \dgfixpoint$ denotes the Jacobian of~$\dgfixpoint$ with respect to the second variable.
    Here we use that the control $u$ is defined on all of~$\RR$, cf.~the definition of $\Ucal$ in \Cref{sec:problem_setting}, so that the time variable ranges over the open set~$\RR$.
    Hence, by the implicit function theorem~\cite[Theorem 4.E in Section 4.8]{zeidler91-applied}, for all~$z_0 \in \tilde\dgnonzeroset$ and~$t_0 \in \RR$ there exists a sufficiently small~$\delta > 0$ and a twice continuously differentiable mapping
    \begin{equation*}
        \begin{aligned}
            \Psi \colon \domain(\Psi) &\to \RR^n, \qquad (z, t; \tau) \mapsto w, \\
            \domain(\Psi) &\coloneq
            \big\{
            (z,t;\tau) \in \RR^n \times \RR \times \RR
            ~\big|~
            \norm{(z-z_0, t-t_0, \tau)}_{\RR^{n+2}} < \delta
            \big\}.
        \end{aligned}
    \end{equation*}
    such that $w=\Psi(z,t;\tau)$ is the unique solution of~$\dgfixpoint\big((z, t; \tau), w\big) = 0$ satisfying $\etabar(z,w) \neq 0$, and $\Psi(z,t;0) = z$ for all~$(z,t;0) \in \domain(\Psi)$.
    This shows that \Cref{scheme:dg-qsr} is locally well-posed in the sense of \Cref{def:locally-well-posed}.
    Moreover, the increment function is given by
    \begin{equation*}
        \begin{aligned}
            \phi(z,t;\tau)
            &=
            \gammabar\big(z,\Psi(z, t; \tau)\big)
            \etabar\big(z, \Psi(z, t; \tau)\big)
            +
            \matrixproj_{\etabar(z, \Psi(z, t; \tau))^\perp}
            \fbar\big(z,\Psi(z, t; \tau)\big) \\
            &\quad +
            \Bbar\big(z,\Psi(z, t; \tau)\big) \ubar(t;\tau)
        \end{aligned}
    \end{equation*}
    and, therefore, twice continuously differentiable.

    \paragraph{Symmetry}
    Since the approximations $\fbar$,~$\Bbar$,~$\Dbar$,~$\ellbar$,~$\Wbar$,~$\etabar$ are symmetric, i.e., $\fbar(z,w) = \fbar(w,z)$, etc.~for all~$z,w \in \RR^n$, the functions $\hbar$,~$\gammabar$ and the projections $\matrixproj_{\etabar^\perp}$ inherit this symmetry, i.e., $\hbar(z,w) = \hbar(w,z)$, $\gammabar(z,w) = \gammabar(w,z)$ and $\matrixproj_{\etabar(z,w)^\perp} = \matrixproj_{\etabar(w,z)^\perp}$.
    Moreover, by assumption, the time-symmetry condition $\ubar(t;\tau) = \ubar(t+\tau;-\tau)$ holds.
    Hence, the map
    \begin{equation*}
        H(z,w,t;\tau) \coloneq \gammabar(z,w)\etabar(z,w) + \matrixproj_{\etabar(z,w)^\perp}\fbar(z,w) + \Bbar(z,w)\ubar(t;\tau)
    \end{equation*}
    satisfies $H(w,z,t+\tau;-\tau) = H(z,w,t;\tau)$ for all admissible $(z,w,t;\tau)$.
    Let now $w = \Psi(z,t;\tau)$, i.e., $w - z = \tau H(z,w,t;\tau)$.
    The point $z' \coloneq \Psi(w, t+\tau; -\tau)$ is characterized by $z' - w = -\tau H(w,z',t+\tau;-\tau)$.
    Since $z - w = -\tau H(w,z,t+\tau;-\tau)$, the uniqueness of the solution of the discrete system~\eqref{eq:discr:dg-qsr-discrete-system} (cf.~\Cref{def:locally-well-posed}) implies $z' = z$.
    Thus, the symmetry condition~\eqref{eq:symmetry-condition} holds.

    \paragraph{Order}
    By the first part of the proof, \Cref{scheme:dg-qsr} is a locally well-posed one-step scheme with twice continuously differentiable one-step map and increment function, cf.~\Cref{def:locally-well-posed}.
    Moreover, the scheme is consistent in the sense of~\eqref{eq:app:consistency-increment}: for all $(z,t)\in\tilde\dgnonzeroset\times\RR$, the increment function satisfies $\phi(z, t; 0) = f(z) + \B(z) u(t)$, which follows directly from the consistency properties of the approximations, i.e., $\fbar(z,z) = f(z)$, $\Bbar(z,z) = \B(z)$, $\ellbar(z,z) = \ell(z)$, $\Wbar(z,z) = W(z)$, $\Dbar(z,z) = \D(z)$, $\etabar(z,z) = \eta(z)$ and $\ubar(t;0) = u(t)$, together with the identity~\eqref{eq:discr:f-qsr-first-term}.
    By assumption, the exact solution $z \in C^3([0,T],\RR^n)$ of~\eqref{eq:moylan-system} satisfies $\eta(z(t)) \neq 0$ for all $t \in [0,T]$, so that the trajectory stays in the domain of well-posedness of the scheme, cf.~\Cref{def:locally-well-posed}; hence, the hypotheses of \Cref{lem:order-symmetric-methods} are satisfied along the trajectory.
    Together with the symmetry established above, \Cref{lem:order-symmetric-methods} yields that the scheme is second-order consistent in the sense of \Cref{def:order}.
    Finally, by the stability argument in \Cref{rem:global-error}, the scheme is second-order accurate in the sense of \Cref{def:accuracy}.
\end{proof}

Before we show that solutions of \Cref{scheme:dg-qsr} satisfy a discrete power balance, let us make a few remarks.
First, note that another second-order accurate approximation of the dynamics~\eqref{eq:moylan-system} is given by
\begin{equation}\label{eq:discr:qsr-implicit-midpoint}
    z_{i+1} = z_i - (t_{i+1} - t_i) (\fbar(z_i, z_{i+1}) + \Bbar(z_i, z_{i+1}) \ubar(t_i;\tau_i)),
\end{equation}
where $\fbar$,~$\Bbar$,~$\ubar$ are chosen based on the implicit midpoint rule, see, e.g.,~\cite[Example~6.3.1]{deuflhard02-scientific}.
To be able to exploit the properties of the discrete gradient, \Cref{scheme:dg-qsr} incorporates~$\etabar$ with a prefactor that accounts for the correct terms in the power balance after the mean value property~\eqref{eq:discr:dg-mean-value-property} is applied.
This is done in such a way that the approximate dynamics~\eqref{eq:discr:qsr-implicit-midpoint} remain unaltered in the direction orthogonal to~$\etabar(z_i, z_{i+1})$.
Hence, in this special case, i.e., when the approximations $\fbar$,~$\Bbar$,~$\Dbar$,~$\ellbar$,~$\Wbar$,~$\ubar$ are chosen as the implicit midpoint approximations, we can view~\eqref{eq:discr:dg-qsr-discrete-system} as the smallest deviation from the implicit midpoint rule~\eqref{eq:discr:qsr-implicit-midpoint} that enables the use of the mean value property~\eqref{eq:discr:dg-mean-value-property} to obtain a time-discrete analogue of the power balance~\eqref{eq:discr:qsr-powerbalance}.
Indeed, in this case the scheme~\eqref{eq:discr:dg-qsr-discrete-system} agrees with the implicit midpoint rule~\eqref{eq:discr:qsr-implicit-midpoint} in all directions orthogonal to $\etabar(z_i, z_{i+1})$, and the only modification consists in replacing the component of $\fbar(z_i, z_{i+1})$ in the direction of $\etabar(z_i, z_{i+1})$, i.e., $\matrixproj_{\etabar(z_i,z_{i+1})}\fbar(z_i,z_{i+1})$, by $\gammabar(z_i,z_{i+1}) \etabar(z_i, z_{i+1})$, where the prefactor $\gammabar(z_i,z_{i+1})$ is chosen exactly such that multiplying~\eqref{eq:discr:dg-qsr-discrete-system} by $\etabar(z_i, z_{i+1})\tp$ and applying the mean value property~\eqref{eq:discr:dg-mean-value-property} yields the desired time-discrete power balance.

Second, we note that even for linear dynamics,~\eqref{eq:discr:dg-qsr-discrete-system} is a nonlinear system, which is suboptimal from a computational viewpoint.
However, as long as the storage function is quadratic, the implicit midpoint rule can be applied for linear systems to obtain a time-discrete power balance.
Third, although the restriction~$\etabar(z_i, z_{i+1}) \neq 0$ is unsatisfactory from a theoretical point of view, we highlight that it did not pose any difficulties in our numerical experiments.

The following proposition shows that solutions of \Cref{scheme:dg-qsr} satisfy a discrete version of the power balance~\eqref{eq:discr:qsr-powerbalance}.

\begin{proposition}[discrete power balance]\label{prop:discr:dg-qsr-powerbalance}
    Let \Cref{as:qsr-scheme} hold and~$(z_i)_{i=0,\dots,\ntimesteps}$ be a solution of \Cref{scheme:dg-qsr}.
    Then, for~$i=0,\dots,\ntimesteps-1$ we have
    \begin{equation}\label{eq:discr:dg-qsr-powerbalance}
        \frac{\ham(z_{i+1}) - \ham(z_{i})}{t_{i+1} - t_{i}}
        =
        s(\ubar(t_i;\tau_i), \ybar_i)
        -
        \norm{\ellbar(z_i, z_{i+1}) + \Wbar(z_i, z_{i+1}) \ubar(t_i;\tau_i)}^2,
    \end{equation}
    where the discrete output~$\ybar_i$ is defined as~$\ybar_i \coloneq \hbar(z_i, z_{i+1}) + \Dbar(z_i, z_{i+1}) \ubar(t_i;\tau_i)$.
\end{proposition}

\begin{proof}
    For the sake of brevity, for~$i=0,\dots,\ntimesteps-1$ and~$\phi \in \{ \fbar, \Bbar, \hbar, \Dbar, \ellbar, \Wbar, \etabar, \gammabar\}$, let us abbreviate~$\phi_i \coloneq \phi(z_i, z_{i+1})$ as well as $\ubar_i\vcentcolon=\ubar(t_i;\tau_i)$.
    Since~$\etabar_i\tp \matrixproj_{\etabar_i^\perp} = 0$, multiplying~\eqref{eq:discr:dg-qsr-discrete-system} by~$\etabar_i\tp$ from the left leaves only the numerator of~$\gammabar_i$ and the input term on the right-hand side.
    Together with the mean value property~\eqref{eq:discr:dg-mean-value-property} for the left-hand side, we arrive at
    \begin{equation}\label{eq:discr:qsr-dg-ham-difference}
        \frac{\ham(z_{i+1}) - \ham(z_{i})}{t_{i+1} - t_{i}}
        =
        \hbar_i\tp \Qbf \hbar_i - \ellbar_i\tp \ellbar_i + \etabar_i\tp \Bbar_i \ubar_i.
    \end{equation}
    With the definition of the supply rate~\eqref{eq:hill-moylan-qsr-supply} and~\eqref{eq:discr:qsr-dg-ham-difference}, we obtain
    \begin{align*}
        & s(\ubar_i, \ybar_i) - \frac{\ham(z_{i+1}) - \ham(z_{i})}{t_{i+1} - t_{i}}
    \\
    = ~ & \ybar_i\tp \Qbf \ybar_i
        + 2 \ybar_i\tp \Sbf \ubar_i
        + \ubar_i\tp \Rbf \ubar_i
        - \hbar_i\tp \Qbf \hbar_i
        + \ellbar_i\tp \ellbar_i
        - \etabar_i\tp \Bbar_i \ubar_i
    \\
    = ~ &
    (\hbar_i + \Dbar_i \ubar_i)\tp \Qbf (\hbar_i + \Dbar_i \ubar_i)
    + 2 (\hbar_i + \Dbar_i \ubar_i)\tp \Sbf \ubar_i
    + \ubar_i\tp \Rbf \ubar_i
    \\*
    &
    - \hbar_i\tp \Qbf \hbar_i
    + \ellbar_i\tp \ellbar_i
    - \etabar_i\tp \Bbar_i \ubar_i
    \\
    = ~ &
    2 \hbar_i\tp (\Qbf \Dbar_i + \Sbf) \ubar_i
    + \ubar_i\tp (\Rbf + \Dbar_i\tp \Sbf + \Sbf\tp \Dbar_i + \Dbar_i\tp \Qbf \Dbar_i) \ubar_i
    + \ellbar_i\tp \ellbar_i
    - \etabar_i\tp \Bbar_i \ubar_i
    \\
    = ~ &
    2 \ellbar_i\tp \Wbar_i \ubar_i
    + \ubar_i\tp \Wbar_i\tp \Wbar_i \ubar_i
    + \ellbar_i\tp \ellbar_i
    =
    \norm{\ellbar_i + \Wbar_i \ubar_i}^2,
    \end{align*}
    where in the last line we used~\eqref{eq:discr:qsr-dg-hbar-definition}, which implies
    \begin{equation*}
        2 \hbar_i\tp (\Qbf \Dbar_i + \Sbf) \ubar_i
        = (\etabar_i\tp \Bbar_i + 2 \ellbar_i\tp \Wbar_i) \ubar_i
        = \etabar_i\tp \Bbar_i \ubar_i + 2 \ellbar_i\tp \Wbar_i \ubar_i.
    \end{equation*}
    Rearranging yields~\eqref{eq:discr:dg-qsr-powerbalance}.
\end{proof}

\section{Numerical Experiments}\label{sec:numerical_experiments}
In this section, we illustrate our theoretical results using numerical experiments.
We consider four nonlinear systems: an optimal control problem, the nonlinear pendulum, a proportional integral (PI) controller, and a synthetic example.

\subsection{Implementation details}
The details of our implementation are as follows.

\paragraph{Time discretization}
In all numerical experiments, we use equidistant time discretization points~$t_j$.
That is, for~$\ntimesteps \in \mathbb{N}$ we set~$\tau = T/\ntimesteps$ and consider~$t_i = i \tau$ for~$i=0,\dots,\ntimesteps$.
Moreover, for the approximations in \Cref{scheme:dg-qsr}, we choose the implicit midpoint rule for $\fbar$,~$\Bbar$,~$\Dbar$,~$\ellbar$,~$\Wbar$, the Gonzalez discrete gradient for $\etabar$, and the trapezoidal rule for $\ubar$.

\paragraph{Nonlinear systems}
To solve the arising nonlinear systems, we use Newton's method.
The derivatives required are computed automatically using \textsf{JAX}~\cite{bradbury18-jax} and the linear systems are solved using \textsf{Lineax}~\cite{rader23-lineax}.
In all but the first time step, the numerical solution of the previous time step is used as the starting value.
We perform ten Newton steps, which leads to a sufficiently small residual for all our examples.

\paragraph{Convergence}
For the convergence analysis, we set~$\tau_{\min} \coloneq 10^{-3}$ and pick the other tested stepsizes~$\tau$ as~$2^s \tau_{\min}$ for some~$s \in \NN$.
A~reference solution $z_{\text{ref}}$ is computed with the implicit midpoint method~\eqref{eq:discr:qsr-implicit-midpoint} using the stepsize~$2^{-3} \tau_\text{min}$.
For the computation of the reference solution, the control $u(\frac{t_{i} + t_{i+1}}{2})$ is approximated by $\frac{u(t_{i}) + u(t_{i+1})}{2}$.
We then compute the relative error
\begin{equation}\label{eq:relative-error}
    \errorstatenodal(\tau)
    \coloneq
    \frac
    {\max_{i} \norm{z_{\text{ref}}(t_{i}) - z_{i}}}
    {\max_{j} \norm{z_{\text{ref}}(t_{j})}}.
\end{equation}

\paragraph{Power balance}
To numerically verify the time-discrete power balance~\eqref{eq:discr:dg-qsr-powerbalance}, we compute the solution of \Cref{scheme:dg-qsr} using $\ntimesteps+1 = 1001$ time points and visualize the relative error
\begin{equation}\label{eq:dg-qsr-power-balance-error}
    \errorpower(t_i)
    \coloneqq
    \frac{
        \left|
            \frac{\ham(z_{i+1}) - \ham(z_{i})}{t_{i+1} - t_{i}}
            +
            \norm{\ellbar(z_i, z_{i+1}) + \Wbar(z_i, z_{i+1}) \ubar_i}^2
            -
            s(\ubar_i, \ybar_i)
        \right|
    }
    {
        \max_{j = 0,\dots,\ntimesteps-1} \left| \frac{\ham(z_{j+1}) - \ham(z_{j})}{t_{j+1} - t_{j}} \right|
    }
    .
\end{equation}

\subsection{Examples}

Our first example is the well-known nonlinear pendulum, see e.g.~\cite[Sec.~3.3]{baker05-pendulum}.

\begin{example}[pendulum]\label{ex:pendulum}
    Denote the angular displacement of the pendulum by~$\theta \in \RR$ and consider
    \begin{equation*}
        \ddot \theta = - g \sin(\theta) - \lambda \dot\theta + u,
    \end{equation*}
    where~$g \in \RR$ is the gravitational constant,~$\lambda \geq 0$ is a friction coefficient and~$u\in C(\RR)$ is a control input influencing the acceleration.
    With~$z(t) = (z_1(t), z_2(t)) = (\theta(t), \dot\theta(t)) \in \RR^2$, the system is associated with the energy
    \begin{equation*}
        \ham(z) = g (1 - \cos(z_1)) + \tfrac12 z_2^2
    \end{equation*}
    and can be expressed as the \portHamiltonian~\cite{vanderschaft14-port} system
    \begin{equation}\label{eq:pendulum}
        \dot{z}
        =
        \vec{
            z_2 \\
            - g \sin(z_1) - \lambda z_2 + u
        }
        =
        (J - R) \eta(z) + B u,
    \end{equation}
    where
    \begin{equation*}
        J = \vec{0 & 1 \\ -1 & 0} \in \RR^{2,2}
        ,~~
        R = \vec{0 & 0 \\ 0 & \lambda} \in \RR^{2,2}
        ,~~
        B = \vec{0 \\ 1} \in \RR^{2,1}.
    \end{equation*}
    The collocated output of the system is the velocity~$y = B\tp \eta(z) = z_2$.
    In particular, the system is passive.
    Moreover, the system is dissipative w.r.t.~the quadratic supply rate~\eqref{eq:hill-moylan-qsr-supply} defined by~$\Qbf = - \lambda$,~$\Sbf = \frac12$ and~$\Rbf = 0$.
    For this choice of the supply rate, the function~$\ell$ from~\eqref{eq:moylan-qsr-conditions} is given by~$\ell = 0$.
    Further, since~$\D = \Rbf = 0$ in~\eqref{eq:moylan-system}, we have~$W = 0$ in~\eqref{eq:moylan-qsr-conditions}.
\end{example}

The next example is based on an optimal control problem and is taken from~\cite[Section 3.5]{vanderschaft17-l2gain}.

\begin{example}[{optimal control problem,~\cite[Section 3.5]{vanderschaft17-l2gain}}]\label{ex:qsr-hjb}
    Consider system~\eqref{eq:moylan-system} with~$\D = 0$ and assume that the value function~$\hamc$, which is defined by
    \begin{equation}\label{eq:qsr-ocp}
        \hamc(z_0) \coloneq \inf_{u \in \Uad(z_0)} \left\{ \frac{1}{2} \int_{0}^{\infty} \norm{y}^2 + \norm{u}^2 \dt \right\}
    \end{equation}
    subject to~\eqref{eq:moylan-system}, is continuously differentiable.
    Here, we consider
    \begin{equation*}
        \Uad(z_0) \coloneq \{ u \in \lebesgue^{2}([0,\infty), \RR^m) ~|~ \text{$u$ renders the integral in~\eqref{eq:qsr-ocp} finite}\}.
    \end{equation*}
    The Hamilton--Jacobi--Bellman equation~\cite{crandall83-viscosity} associated with the optimal control problem then reads
    \begin{equation*}
        \etac(z)\tp f(z) - \tfrac12 \norm{\B(z)\tp \etac(z)}^2 + \tfrac12 \norm{y}^2 = 0.
    \end{equation*}
    In particular, if we consider the output~$\yc \coloneq \B(z)\tp \etac(z)$, we see that the dynamics are dissipative with respect to the supply rate
    \begin{equation*}
        s(u, \yc) = \tfrac12 \norm{\yc}^2 + \yc\tp u = \yc\tp \Qbf \yc + 2 \yc\tp \Sbf u,
    \end{equation*}
    where we set~$\Qbf = \Sbf = \tfrac12 I_m$.
    For this example, the storage function is~$\hamc$ and the function~$\ell$ from~\eqref{eq:moylan-qsr-conditions} is given by~$\ell(z) = \frac{1}{\sqrt{2}} y =  \frac{1}{\sqrt{2}} h(z)$.
    Further, since~$\D = \Rbf = 0$, we have~$W = 0$ in~\eqref{eq:moylan-qsr-conditions}.
\end{example}

The next example is the \emph{proportional-integral (PI)} controller.
Such controllers are widely used in practice~\cite{astrom93-automatic}.

\begin{example}[PI controller]\label{ex:pi-controller}
    As discussed in~\cite[Example 4.1.4]{vanderschaft17-l2gain}, the proportional-integral controller for the signal~$u\colon \RR \to \RR$ described by
    \begin{equation}\label{eq:pi-controller}
        \begin{aligned}
            \dot{z} & = u, ~~ z(0) = z_0\\
            y & = k_I z + k_P u
        \end{aligned}
    \end{equation}
    with the integral and proportional coefficients~$k_I \geq 0$ and~$k_P \geq 0$, respectively, is dissipative w.r.t.~the quadratic supply rate~\eqref{eq:hill-moylan-qsr-supply} defined by~$\Qbf = 0$,~$\Sbf = \frac12$ and~$\Rbf = - k_P$ with the storage function~$\ham(z) = \tfrac12 k_I z^2$.
    For this example, we have~$\ell = 0$ and~$W = 0$ in~\eqref{eq:moylan-qsr-conditions}.
\end{example}

Our final example is purely synthetic.
The example is taken from \cite[Case~3 in Section~V]{hill80-connections}, but we changed the sign of $f(z)$ to render the system stable, and accordingly the sign of the nonlinearity in the output.

\begin{example}[{modified from~\cite{hill80-connections}}]\label{ex:qsr-hill-moylan}
    Consider the one-dimensional system
    \begin{equation}\label{eq:qsr-hill-moylan-system}
        \begin{aligned}
            \dot{z} & = - z - \frac{\alpha z}{1 + z^4} + 2 \lambda u, ~~ z(0) = z_0 \\
            y & = - \frac{\alpha z}{1 + z^4} + \lambda u,
        \end{aligned}
    \end{equation}
    where~$z_0 \in \RR$ is the initial value and~$\alpha > 0$ and~$\lambda \neq 0$ are parameters.
    For~$\ham(z) = \tfrac{\alpha}{2} \arctan(z^2)$, we have~$\ham'(z) = \tfrac{\alpha z}{1 + z^4}$ and hence along trajectories of~\eqref{eq:qsr-hill-moylan-system} it holds that
    \begin{align*}
        \ham(z(t_1)) - \ham(z(t_0))
        & = \int_{t_0}^{t_1} \ham'(z) \dot{z} \dt
        \\
        & = \int_{t_0}^{t_1} - \frac{\alpha z^2}{1 + z^4} - \frac{\alpha^2 z^2}{(1 + z^4)^2} + \frac{2 \lambda \alpha z u}{1 + z^4} \dt
        \\
        & = \int_{t_0}^{t_1} - \frac{\alpha z^2}{1 + z^4} + \lambda^2 u^2 - y^2 \dt.
    \end{align*}
    In particular, the system~\eqref{eq:qsr-hill-moylan-system} is dissipative w.r.t.~the quadratic supply rate~\eqref{eq:hill-moylan-qsr-supply} defined by~$\Qbf = -1$,~$\Sbf = 0$ and~$\Rbf = \lambda^2$.
    For this example, we have~$\ell(z) = \frac{\sqrt{\alpha} z}{\sqrt{1+z^4}}$ and~$W = 0$ in~\eqref{eq:moylan-qsr-conditions}.
\end{example}

\subsection{Results}
For the numerical experiments, we consider the following settings.
\begin{itemize}
\item
    For the pendulum (\Cref{ex:pendulum}), we choose the parameters~$g=9.81$,~$\lambda = 0.2$, the initial condition~$z_0 = \textvec{\pi/4 \\ -1}$, and the control input~$u(t) = \sin(2t)$.
\item For the infinite-horizon optimal control problem~\eqref{eq:qsr-ocp} (\Cref{ex:qsr-hjb}), we consider~\eqref{eq:moylan-system} to be the linear time-invariant system%
    \footnote{The proposed discrete gradient method also works for~\Cref{ex:qsr-hjb} when~\eqref{eq:moylan-system} is allowed to be nonlinear. However, in that case the convergence of the iterates is deteriorated due to errors in the approximation of~$\etac(z)$. This is why we restrict ourselves to linear systems here.}
    \begin{equation}\label{eq:lti-test}
        \begin{aligned}
            \dot{z} & = A z + B u, ~~ z(0) = z_0, \\
            y & = C z
        \end{aligned}
    \end{equation}
    with~$A = \textvec{0.1 & 1 \\ -1 & 0.1}$,~$B = \textvec{0 \\ 1}$,~$C = \textvec{ 1 & 0}$, and the initial condition~$z_0 = \textvec{ 1 \\ 1 }$.
    \newcommand{\Pc}{P_\mathrm{c}}
    Then the value function~$\hamc$ of~\eqref{eq:qsr-ocp} is given by~$\hamc(z_0) = \tfrac12 z_0\tp \Pc z_0$, where~$\Pc = \Pc\tp \succ 0$ is the stabilizing solution to the algebraic Riccati equation
    \begin{equation*}
        A\tp \Pc + \Pc A - \Pc B B\tp \Pc + C\tp C = 0.
    \end{equation*}
    Consequently, we have~$\yc = \B(z)\tp \etac(z) = B\tp \Pc z$.
    As the control input, we consider~$u(t) = \sin({t^2}/{4})$.
\item For the PI controller (\Cref{ex:pi-controller}), we choose the parameters~$k_I = k_P = 1$, initial condition~$z_0 = 1$, and control input~$u(t) = t^2 \exp(-t)$.
\item For the synthetic example (\Cref{ex:qsr-hill-moylan}), we choose the parameters~$\lambda = 1$,~$\alpha = 2$, initial condition~$z_0 = 1$, and control input~$u(t) = \exp(-(t-4)^2) + \exp(-(t-7)^2)$.
\end{itemize}

In \Cref{fig:all_convergence} the convergence of the solution \Cref{scheme:dg-qsr} is shown for all our examples.
We observe that the scheme is second order accurate in all cases.
In \Cref{fig:all_balances}, the relative error in~\eqref{eq:discr:dg-qsr-powerbalance} is shown.
We observe that the errors are close to machine precision throughout the time horizon.
The deviation from machine precision can be explained by the fact that an approximate solution of a nonlinear equation system has to be determined in every time step and the corresponding approximation error affects the accuracy of the power balance.

\begin{figure}
    \centering
	\begin{subfigure}[t]{.45\textwidth}
        \centering
        \scalebox{0.48}{\input{all_convergence.pgf}}
        \caption{
            Convergence to the reference solution.
        }
        \label{fig:all_convergence}
    \end{subfigure}
	\hfill
	\begin{subfigure}[t]{.45\textwidth}
		\centering
        \scalebox{0.48}{\input{all_balances.pgf}}
        \caption{
            Relative error in~\eqref{eq:discr:dg-qsr-powerbalance} over the time horizon.
        }
        \label{fig:all_balances}
	\end{subfigure}
    \caption{
        Numerical results for \Cref{ex:pendulum,ex:qsr-hjb,ex:pi-controller,ex:qsr-hill-moylan}.
    }
\end{figure}

\section{Concluding remarks}
In this paper, we propose a discrete gradient scheme for a special class of QSR-dissipative systems with equal input and output numbers.
In a special case, the new scheme can be viewed as the smallest deviation from the implicit midpoint rule that incorporates discrete gradients and ensures a time-discrete power balance.
We provide sufficient conditions for the scheme to be locally well-posed and second-order consistent, which implies second-order accuracy of the numerical solution by a stability argument.
Moreover, we illustrate our theoretical results by means of numerical experiments.
As expected, discrete solutions satisfy energy and power balances and the corresponding dissipation inequalities up to the accuracy of the nonlinear system solver.
Furthermore, we show second-order accuracy for all examples.

Topics for future research include extensions to higher-order schemes as well as the extension to more general QSR-dissipative systems with unequal input and output numbers.
Moreover, the structure of QSR-dissipative systems identified on the time-continuous level might also pave the way for the development of structure-preserving model order reduction techniques for large-scale QSR-dissipative systems.

\paragraph{Acknowledgment}
A.\,K. thanks the Deutsche Forschungsgemeinschaft for their support within the subproject B03 in the Sonderforschungsbereich/Transregio 154 “Mathematical Modelling, Simulation and Optimization using the Example of Gas Networks” (Project 239904186).
P.\,S. thanks the Deutsche Forschungsgemeinschaft for their support within the Sonderforschungsbereich 1294 “Data Assimilation -- The Seamless Integration of Data and Models” (Project 318763901).

\paragraph{Code availability}
All custom code used to generate the results in this paper is available in the public \textsf{GitHub} repository \url{https://github.com/akarsai/qsr-discrete-gradients}.

\appendix
\renewcommand{\thetheorem}{\Alph{section}.\arabic{theorem}}
\renewcommand{\thedefinition}{\Alph{section}.\arabic{theorem}}
\renewcommand{\thelemma}{\Alph{section}.\arabic{theorem}}
\renewcommand{\theremark}{\Alph{section}.\arabic{theorem}}
\setcounter{theorem}{0}
\section{One-step schemes: local well-posedness, consistency, and order}\label{app:one-step}

We collect general results on one-step schemes for the non-autonomous system~\eqref{eq:moylan-system}, which are used in the main text to analyze \Cref{scheme:dg-qsr}.
The results are adapted from the standard literature, see, e.g.,~\cite[Chapter~4]{deuflhard02-scientific} and~\cite[Theorem~3.2 in Section~II.3]{hairer06-geometric}; here, we present self-contained, problem-adapted versions and allow for open domains for the state space.
Especially, instead of the sets $\lbrace z\in\RR^n\mid\eta(z)\ne 0\rbrace$ and $\lbrace (z,w)\in\RR^n\times\RR^n \mid\etabar(z,w)\ne 0\rbrace$ considered in \Cref{sec:time_discretization}, we consider in the following the more general open sets $\tilde\dgnonzeroset\subseteq \RR^n$ and $\dgnonzeroset\in\RR^n\times\RR^n$ satisfying
\begin{equation*}
	z\in\tilde{\dgnonzeroset} \quad\Leftrightarrow\quad (z,z)\in \dgnonzeroset.
\end{equation*}

A \emph{one-step scheme} for~\eqref{eq:moylan-system} is given by a map $\dgfixpoint\colon \domain(\dgfixpoint)\subseteq\RR^{n+2}\times\RR^n\to\RR^n$ with $\domain(\dgfixpoint)=\lbrace ((z,t;\tau),w)\in\RR^{n+2}\times\RR^n\mid(z,w)\in \dgnonzeroset\rbrace$ such that the time-discrete solution is defined by
\begin{equation}
	\label{eq:app:one-step-scheme}
	\begin{aligned}
		z_0 &= z(0)\in \tilde\dgnonzeroset,\\
		\dgfixpoint((z_i,t_i;\tau_i),z_{i+1})&=0\quad \text{for }i=0,\ldots,q-1.
	\end{aligned}
\end{equation}
Moreover, we assume $\dgfixpoint((z,t;0),z)=0$ for all $(z,t)\in\tilde\dgnonzeroset\times\RR$.

The following definition formalizes the notion of local well-posedness of the one-step scheme.

\begin{definition}[locally well-posed]\label{def:locally-well-posed}
    A one-step scheme in the sense of~\eqref{eq:app:one-step-scheme} is called \emph{locally well-posed} if for each $(z_0,t_0)\in \tilde\dgnonzeroset\times\RR$ there exists a $\delta>0$ such that for each $\tau\in(-\delta,\delta)$ the equation $\dgfixpoint((z_0,t_0;\tau),w)=0$ is uniquely solvable for $w$.
    Since this solution depends on $z_0$, $t_0$, and $\tau$, we denote it by $\Psi(z_0,t_0;\tau)$ and call $\Psi$ the \emph{one-step map} of the scheme.
    Moreover, the associated \emph{increment function} $\phi$ is defined by
    \begin{equation}\label{eq:app:increment-function}
    	\phi(z, t; \tau) \coloneq \frac{\Psi(z, t; \tau) - z}{\tau} \qquad \text{for } \tau \neq 0
    \end{equation}
    and by its continuous extension to $\tau = 0$, whenever this extension exists.
\end{definition}

Next, we introduce the accuracy notions for one-step schemes.

\begin{definition}[consistency and order]\label{def:order}
    For $(z_0,t_0)\in \tilde\dgnonzeroset\times\RR$, let $\Psi$ be the one-step map of a one-step scheme in the sense of \Cref{def:locally-well-posed} and let $z$ be the exact solution of~\eqref{eq:moylan-system} with $z(t_0) = z_0$ and control $u$.
    Then the \emph{local truncation error} of the scheme at $(z_0, t_0)$ is defined as
    \begin{equation}\label{eq:local-truncation-error}
        \rho(z_0, t_0; \tau) \coloneq z(t_0 + \tau) - \Psi(z_0, t_0; \tau).
    \end{equation}
    The scheme is called
    \begin{itemize}
    	\item \emph{consistent} if $\norm{\rho(z_0, t_0; \tau)} = o(\tau)$ as $\tau \to 0$ for all $(z_0,t_0)\in \tilde\dgnonzeroset\times\RR$, and
    	\item \emph{second-order consistent} if $\norm{\rho(z_0, t_0; \tau)} = O(\tau^3)$ as $\tau \to 0$ locally uniformly in $(z_0,t_0)\in \tilde\dgnonzeroset\times\RR$.
    \end{itemize}
\end{definition}

If the scheme is locally well-posed with a one-step map which is continuously differentiable w.r.t.~$\tau$, then consistency is equivalent to
\begin{equation}\label{eq:app:consistency-increment}
    \phi(z_0, t_0; 0) = f(z_0) + \B(z_0) u(t_0)
\end{equation}
for all $(z_0,t_0)\in \tilde\dgnonzeroset\times\RR$.

Finally, we quantify the accuracy of the time-discrete solution.

\begin{definition}[accuracy]\label{def:accuracy}
    With the notation above, let $z$ denote the exact solution of~\eqref{eq:moylan-system} with $z(0) = z_0$ and control $u$, and let $(z_i)_{i=0,\dots,\ntimesteps}$ denote the corresponding time-discrete solution.
    The scheme is called \emph{second-order accurate} if
    \begin{equation}\label{eq:app:accuracy}
        \max_{i = 0,\dots,\ntimesteps} \norm{ z(t_i) - z_i } = O(\tau^2) \qquad \text{as } \tau = \max_{i=0,\dots,\ntimesteps-1} \tau_i \to 0.
    \end{equation}
\end{definition}

A standard stability argument for one-step methods shows that the second-order consistency implies second-order accuracy, see, e.g.,~\cite[Theorem~4.10]{deuflhard02-scientific}.

\begin{remark}[stability and global error]\label{rem:global-error}
    If the one-step map $\Psi$ is locally
    Lipschitz continuous with respect to the state variable $z$ with Lipschitz constant $1 + L \tau$ for some $L \geq 0$ independent of $\tau$ and the scheme is second-order consistent, then stability implies the second-order accuracy estimate~\eqref{eq:app:accuracy}, provided that the trajectory stays in the domain of well-posedness of the scheme.
    This can be seen by subtracting the one-step equation from the Taylor expansion of the exact solution and applying the discrete Grönwall lemma, which yields~\eqref{eq:app:accuracy}.
\end{remark}

A classical result on the order of symmetric one-step methods states that consistent and symmetric schemes are second-order consistent, see, e.g.,~\cite[Theorem~3.2 in Section~II.3]{hairer06-geometric} for the autonomous case.

\begin{lemma}[order of symmetric methods]\label{lem:order-symmetric-methods}
    Consider a locally well-posed one-step scheme in the sense of \Cref{def:locally-well-posed} and let the increment function $\phi$ be twice continuously differentiable w.r.t.~$\tau$, continuously differentiable for all $(z_0,t_0;0)\in\tilde\dgnonzeroset\times\RR\times\RR$, and consistent, cf.~\eqref{eq:app:consistency-increment}.
    In addition, let the scheme be symmetric, i.e., the \emph{symmetry condition}
    \begin{equation}\label{eq:symmetry-condition}
        \Psi\big( \Psi(z_0, t_0; \tau), t_0 + \tau; -\tau \big) = z_0
    \end{equation}
    is satisfied for all $(z_0, t_0; \tau) \in \tilde\dgnonzeroset\times\RR\times\RR$ for which the left-hand side is well-defined.
    Moreover, assume that, for every $(z_0, t_0) \in \tilde\dgnonzeroset\times\RR$, the exact solution $z$ of~\eqref{eq:moylan-system} with $z(t_0) = z_0$ and control $u$ is three times continuously differentiable on an interval $[t_0, t_0 + \tau_*]$, $\tau_* > 0$, and satisfies $z(t) \in \tilde\dgnonzeroset$ for all $t \in [t_0, t_0 + \tau_*]$.
    Then the scheme is second-order consistent in the sense of \Cref{def:order}.
\end{lemma}

For \Cref{scheme:dg-qsr}, the condition $z(t) \in \tilde\dgnonzeroset$ in \Cref{lem:order-symmetric-methods} is equivalent to $\eta(z(t)) \neq 0$ for all $t\in [t_0, t_0 + \tau_*]$.

\begin{proof}
    Fix $(z_0, t_0) \in \tilde\dgnonzeroset\times\RR$ and let $z$ denote the exact solution of~\eqref{eq:moylan-system} with $z(t_0) = z_0$ and control $u$.
    For the sake of brevity, set $x \coloneq z(t_0)$ and $v(z,t) \coloneq f(z) + \B(z) u(t)$.
    Since, by assumption, the increment function $\phi$ is twice continuously differentiable w.r.t.~$\tau$, and since, by~\eqref{eq:app:increment-function}, the identity $\Psi(z, t; \tau) = z + \tau \phi(z, t; \tau)$ holds, we may expand $\Psi$ as
    \begin{equation*}
        \Psi(x, t_0; \tau) = x + \tau \phi(x, t_0; 0) + \tau^2 \partial_\tau \phi(x, t_0; 0) + O(\tau^3).
    \end{equation*}
    By consistency, we have $\phi(z, t; 0) = v(z, t)$ for all $(z, t) \in \tilde\dgnonzeroset\times\RR$, cf.~\eqref{eq:app:consistency-increment}.
    By Taylor's theorem applied to the exact solution,
    \begin{equation*}
        z(t_0 + \tau) = x + \tau \dot x + \tfrac{\tau^2}{2} \ddot x + O(\tau^3),
    \end{equation*}
    where $\dot x = v(x,t_0)$ and $\ddot x = \partial_z v(x,t_0) v(x,t_0) + \partial_t v(x,t_0)$.
    By $\phi(x, t_0; 0) = v(x,t_0) = \dot x$, it remains to verify that $\partial_\tau \phi(x, t_0; 0) = \tfrac12 \ddot x$.

    To this end, we note that the symmetry condition~\eqref{eq:symmetry-condition} may be expressed in terms of the increment function as
    \begin{equation}\label{eq:symmetry-increment}
        \phi(x, t_0; \tau) = \phi\big( \Psi(x, t_0; \tau), t_0 + \tau; -\tau \big)
    \end{equation}
    for all sufficiently small $\tau$.
    Taking the derivative of this equation w.r.t.~$\tau$ and evaluating the resulting equation at $\tau=0$ yields
    \begin{align*}
        \partial_\tau \phi(x, t_0; 0)
        & =
        \partial_z \phi(x, t_0; 0) \partial_\tau\Psi(x, t_0; 0) - \partial_\tau \phi(x, t_0; 0) + \partial_t \phi(x, t_0; 0)
    \end{align*}
    or, equivalently,
    \begin{equation*}
        2 \partial_\tau \phi(x, t_0; 0)
        = \partial_z \phi(x, t_0; 0) \phi(x, t_0; 0) + \partial_t \phi(x, t_0; 0)
        = \partial_z v(x,t_0) \dot x + \partial_t v(x,t_0)
        = \ddot x,
    \end{equation*}
    where we used $\partial_\tau\Psi(x, t_0; 0) = \phi(x, t_0; 0)$ and $\phi(x, t_0; 0) = v(x, t_0)$ in the second equality.
    Combining the above expansions, the local truncation error~\eqref{eq:local-truncation-error} satisfies
    \begin{equation*}
        \rho(z_0, t_0; \tau)
        = \tau \big( \dot x - v(x, t_0) \big)
        + \tau^2 \left( \tfrac12 \ddot x - \partial_\tau \phi(x, t_0; 0) \right)
        + O(\tau^3)
        = O(\tau^3),
    \end{equation*}
    which proves the assertion.
\end{proof}

\bibliographystyle{siam}
\bibliography{QSRdissip}

\begin{thebibliography}{10}

\bibitem{arnold89-mathematical}
{\sc V.~I. Arnold}, {\em Mathematical Methods of Classical Mechanics}, Springer New York, USA, second~ed., 1989.

\bibitem{astrom93-automatic}
{\sc K.~J. {\AA}ström, T.~Hägglund, C.~C. Hang, and W.~K. Ho}, {\em Automatic tuning and adaptation for {PID} controllers - a survey}, Control Engineering Practice, 1 (1993), pp.~699--714.

\bibitem{baker05-pendulum}
{\sc G.~L. Baker and J.~A. Blackburn}, {\em The Pendulum: A Case Study in Physics}, Oxford University Press, Oxford, UK, 2005.

\bibitem{bradbury18-jax}
{\sc J.~Bradbury, R.~Frostig, P.~Hawkins, M.~J. Johnson, C.~Leary, D.~Maclaurin, G.~Necula, A.~Paszke, J.~Vander{P}las, S.~Wanderman-{M}ilne, and Q.~Zhang}, {\em {JAX}: composable transformations of {P}ython+{N}um{P}y programs}, 2018.

\bibitem{brogliato20-dissipative}
{\sc B.~Brogliato, R.~Lozano, B.~Maschke, and O.~Egeland}, {\em Dissipative Systems Analysis and Control: Theory and Applications}, Springer Nature Switzerland, Cham, Switzerland, third~ed., 2020.

\bibitem{celledoni12-preserving}
{\sc E.~Celledoni, V.~Grimm, R.~I. McLachlan, D.~I. McLaren, D.~O{'}Neale, B.~Owren, and G.~R.~W. Quispel}, {\em Preserving energy resp.~dissipation in numerical {PDE}s using the ``average vector field'' method}, Journal of Computational Physics, 231 (2012), pp.~6770--6789.

\bibitem{celledoni17-energy}
{\sc E.~Celledoni and E.~H. H{\o}iseth}, {\em Energy-preserving and passivity-consistent numerical discretization of port-{Hamiltonian} systems}, arXiv preprint {1706.08621},  (2017).

\bibitem{crandall83-viscosity}
{\sc M.~G. Crandall and P.-L. Lions}, {\em Viscosity solutions of {Hamilton}-{Jacobi} equations}, Transactions of the American Mathematical Society, 277 (1983), pp.~1--42.

\bibitem{deuflhard02-scientific}
{\sc P.~Deuflhard and F.~Bornemann}, {\em Scientific Computing with Ordinary Differential Equations}, Springer, New York, NY, USA, 2002.

\bibitem{egger21-energy}
{\sc H.~Egger, O.~Habrich, and V.~Shashkov}, {\em On the energy stable approximation of {H}amiltonian and gradient systems}, Computational Methods in Applied Mathematics, 21 (2021), pp.~335--349.

\bibitem{ehrhardt25-geometric}
{\sc M.~J. Ehrhardt, E.~S. Riis, T.~Ringholm, and C.-B. Schönlieb}, {\em A geometric integration approach to smooth optimization: foundations of the discrete gradient method}, {IMA} Journal of Numerical Analysis, 45 (2025), pp.~1269--1299.

\bibitem{eidnes22-order}
{\sc S.~Eidnes}, {\em Order theory for discrete gradient methods}, BIT Numerical Mathematics, 62 (2022), pp.~1207--1255.

\bibitem{giesselmann25-energy-consistent}
{\sc J.~Giesselmann, A.~Karsai, and T.~Tscherpel}, {\em Energy-consistent {P}etrov--{G}alerkin time discretization of port-{H}amiltonian systems}, SMAI Journal of Computational Mathematics, 11 (2025), pp.~335--367.

\bibitem{gonzalez96-time}
{\sc O.~Gonzalez}, {\em Time integration and discrete {Hamiltonian} systems}, Journal of Nonlinear Science, 6 (1996), pp.~449--467.

\bibitem{grmela97-dynamics}
{\sc M.~Grmela and H.~C. \"Ottinger}, {\em Dynamics and thermodynamics of complex fluids. {I}. {D}evelopment of a general formalism}, Physical Review E, 56 (1997), pp.~6620--6632.

\bibitem{hairer06-geometric}
{\sc E.~Hairer, G.~Wanner, and C.~Lubich}, {\em Geometric Numerical Integration}, Springer, Berlin, Germany, 2006.

\bibitem{hairer93-solving}
{\sc E.~Hairer, G.~Wanner, and S.~P. N{\o}rsett}, {\em Solving Ordinary Differential Equations I}, Springer, Berlin, G., 1993.

\bibitem{harten83-upstream}
{\sc A.~Harten, P.~D. Lax, and B.~van Leer}, {\em On upstream differencing and {Godunov}-type schemes for hyperbolic conservation laws}, SIAM Review, 25 (1983), pp.~35--61.

\bibitem{hill80-connections}
{\sc D.~Hill and P.~Moylan}, {\em Connections between finite-gain and asymptotic stability}, IEEE Transactions on Automatic Control, 25 (1980), pp.~931--936.

\bibitem{hill75-cyclo}
{\sc D.~J. Hill and P.~J. Moylan}, {\em Cyclo-dissipativeness, dissipativeness, and losslessness for nonlinear dynamical systems}, Technical Report No. EE7526,  (1975).

\bibitem{hirsch74-differential}
{\sc M.~W. Hirsch and S.~Smale}, {\em Differential Equations, Dynamical Systems, and Linear Algebra}, Academic Press, New York, NY, USA, 1974.

\bibitem{itoh88-hamiltonian}
{\sc T.~Itoh and K.~Abe}, {\em Hamiltonian-conserving discrete canonical equations based on variational difference quotients}, Journal of Computational Physics, 76 (1988), pp.~85--102.

\bibitem{juengel19-structure-preserving}
{\sc A.~J\"ungel, U.~Stefanelli, and L.~Trussardi}, {\em Two structure-preserving time discretizations for gradient flows}, Applied Mathematics and Optimization, 80 (2019), pp.~733--764.

\bibitem{juengel22-minimizing-movements}
\leavevmode\vrule height 2pt depth -1.6pt width 23pt, {\em A minimizing-movements approach to {GENERIC} systems}, Mathematics in Engineering, 4 (2022), pp.~1--18.

\bibitem{kemmochi23-higher}
{\sc T.~Kemmochi}, {\em Higher order discrete gradient method by the discontinuous {G}alerkin time-stepping method}, ArXiv preprint 2308.02334v1, 2023.

\bibitem{kinon26-discrete}
{\sc P.~L. Kinon, R.~Morandin, and P.~Schulze}, {\em Discrete gradient methods for port-{H}amiltonian differential-algebraic equations}, Applied Numerical Mathematics, 223 (2026), pp.~45--75.

\bibitem{kotyczka19-numerical}
{\sc P.~Kotyczka}, {\em Numerical Methods for Distributed Parameter Port-Hamiltonian Systems}, TUM.University Press, Munich, Germany, 2019.

\bibitem{kotyczka19-discrete-time}
{\sc P.~Kotyczka and L.~Lef\`{e}vre}, {\em Discrete-time port-{H}amiltonian systems: a definition based on symplectic integration}, Systems \& Control Letters, 133 (2019), p.~104530.

\bibitem{mclachlan99-geometric}
{\sc R.~I. McLachlan, G.~R.~W. Quispel, and N.~Robidoux}, {\em Geometric integration using discrete gradients}, Philosophical Transactions of the Royal Society of London. Series {A}: Mathematical, Physical and Engineering Sciences, 357 (1999), pp.~1021--1045.

\bibitem{mehrmann19-structure-preserving}
{\sc V.~Mehrmann and R.~Morandin}, {\em Structure-preserving discretization for port-{H}amiltonian descriptor systems}, in Proceedings of the 58th IEEE Conference on Decision and Control, Nice, France, 2019, pp.~6863--6868.

\bibitem{morandin24-modeling}
{\sc R.~Morandin}, {\em Modeling and numerical treatment of port-{H}amiltonian descriptor systems}, PhD thesis, Technische Universit\"at Berlin, Germany, 2024.

\bibitem{norton14-discrete}
{\sc R.~A. Norton and G.~R.~W. Quispel}, {\em Discrete gradient methods for preserving a first integral of an ordinary differential equation}, Discrete and Continuous Dynamical Systems, 34 (2014), pp.~1147--1170.

\bibitem{ortega02-interconnection}
{\sc R.~Ortega, A.~van~der Schaft, B.~Maschke, and G.~Escobar}, {\em Interconnection and damping assignment passivity-based control of port-controlled {H}amiltonian systems}, Automatica, 38 (2002), pp.~585--596.

\bibitem{oettinger18-generic}
{\sc H.~C. \"Ottinger}, {\em {GENERIC} integrators: structure preserving time integration for thermodynamic systems}, Journal of Non-Equilibrium Thermodynamics, 43 (2018), pp.~89--100.

\bibitem{oettinger97-dynamics}
{\sc H.~C. \"Ottinger and M.~Grmela}, {\em Dynamics and thermodynamics of complex fluids. {II}. {I}llustrations of a general formalism}, Physical Review E, 56 (1997), pp.~6633--6655.

\bibitem{rader23-lineax}
{\sc J.~Rader, T.~Lyons, and P.~Kidger}, {\em {Lineax}: unified linear solves and linear least-squares in {JAX} and {Equinox}}, AI for science workshop at Neural Information Processing Systems 2023, arXiv:2311.17283,  (2023).

\bibitem{schulze23-structure2}
{\sc P.~Schulze}, {\em Structure-preserving time discretization of port-{H}amiltonian systems via discrete gradient pairs}, in Progress in Industrial Mathematics at ECMI 2023, K.~Burnecki, J.~Szwabi\'nski, and M.~Teuerle, eds., Wroc\l{}aw, Poland, 2026, Springer Nature Switzerland, pp.~501--511.

\bibitem{vanderschaft17-l2gain}
{\sc A.~{Van der Schaft}}, {\em $L^2$-Gain and Passivity Techniques in Nonlinear Control}, vol.~2, Springer International Publishing, 2017.

\bibitem{vanderschaft14-port}
{\sc A.~J. {Van der Schaft} and D.~Jeltsema}, {\em {Port-Hamiltonian} systems theory: An introductory overview}, Foundations and Trends in Systems and Control, 1 (2014), pp.~173--378.

\bibitem{wang03-generalized}
{\sc Y.~Wang, C.~Li, and D.~Cheng}, {\em Generalized {Hamiltonian} realization of time-invariant nonlinear systems}, Automatica, 39 (2003), pp.~1437--1443.

\bibitem{willems72-dissipative1}
{\sc J.~Willems}, {\em Dissipative dynamical systems part {I}: general theory}, Archive for Rational Mechanics and Analysis, 45 (1972), pp.~321--351.

\bibitem{zeidler91-applied}
{\sc E.~Zeidler}, {\em Applied functional analysis: main principles and their applications}, vol.~109, Springer, New York, NY, USA, 1991.

\end{thebibliography}

\end{document}